\documentclass[11pt, reqno]{amsart}
\pdfoutput=1

\usepackage{amsmath, amsthm, amssymb}
\usepackage{fullpage}
\usepackage{enumerate}
\usepackage{ifpdf}
\usepackage{mathrsfs}
\ifpdf
\usepackage[pdftex]{graphicx}
\else
\usepackage[dvips]{graphicx}
\fi
\usepackage{tikz}
 	 \usetikzlibrary{arrows,backgrounds}
\usepackage[all]{xy}

\input xy
\xyoption{all}

\usepackage[pdftex,plainpages=false,hypertexnames=false,pdfpagelabels]{hyperref}
\newcommand{\arXiv}[1]{\href{http://arxiv.org/abs/#1}{\tt arXiv:\nolinkurl{#1}}}
\newcommand{\arxiv}[1]{\href{http://arxiv.org/abs/#1}{\tt arXiv:\nolinkurl{#1}}}

\newcommand{\doi}[1]{\href{http://dx.doi.org/#1}{{\tt DOI:#1}}}

\newcommand{\googlebooks}[1]{(preview at \href{http://books.google.com/books?id=#1}{google books})}

\usepackage{xcolor}
\definecolor{dark-red}{rgb}{0.7,0.25,0.25}
\definecolor{dark-blue}{rgb}{0.15,0.15,0.55}
\definecolor{medium-blue}{rgb}{0,0,.8}
\definecolor{DarkGreen}{RGB}{0,150,0}
\hypersetup{
   colorlinks, linkcolor={purple},
   citecolor={medium-blue}, urlcolor={medium-blue}
}

\usepackage{fullpage}
\theoremstyle{plain}
\newtheorem{thm}{Theorem}[section]
\newtheorem*{thm*}{Theorem}
\newtheorem{thmalpha}{Theorem}

\newtheorem{cor}[thm]{Corollary}
\newtheorem{coralpha}[thmalpha]{Corollary}
\newtheorem*{cor*}{Corollary}
\newtheorem{lem}[thm]{Lemma}

\newtheorem{prop}[thm]{Proposition}

\newtheorem*{quest*}{Question}

\theoremstyle{definition}
\newtheorem{defn}[thm]{Definition}
\newtheorem{assumption}[thm]{Assumption}
\newtheorem{assumptions}[thm]{Assumptions}
\newtheorem{nota}[thm]{Notation}

\newtheorem{exs}[thm]{Examples}
\newtheorem{ex}[thm]{Example}
\newtheorem{rem}[thm]{Remark}
\newtheorem{rems}[thm]{Remarks}
\newtheorem{fact}[thm]{Fact}

\DeclareMathOperator{\conv}{conv}

\DeclareMathOperator{\dist}{dist}

\DeclareMathOperator{\Gr}{Gr}

\DeclareMathOperator{\OP}{op}

\DeclareMathOperator{\spann}{span}

\DeclareMathOperator{\Tr}{Tr}
\DeclareMathOperator{\tr}{tr}

\newcommand{\comment}[1]{}

\newcommand{\be}{\begin{enumerate}[(1)]}
\newcommand{\ee}{\end{enumerate}}

\newcommand{\N}{\mathbb{N}}

\newcommand{\e}{\epsilon}

\newcommand{\I}{\infty}
\newcommand{\set}[2]{\left\{#1 \middle| #2\right\}}

\newcommand{\op}{^{\OP}}
\newcommand{\noshow}[1]{}
\renewcommand{\MR}[1]{}

\newcommand{\Asterisk}{\mathop{\scalebox{1.5}{\raisebox{-0.2ex}{$*$}}}}

\def\semicolon{;}
\def\applytolist#1{
    \expandafter\def\csname multi#1\endcsname##1{
        \def\multiack{##1}\ifx\multiack\semicolon
            \def\next{\relax}
        \else
            \csname #1\endcsname{##1}
            \def\next{\csname multi#1\endcsname}
        \fi
        \next}
    \csname multi#1\endcsname}

\def\calc#1{\expandafter\def\csname c#1\endcsname{{\mathcal #1}}}
\applytolist{calc}QWERTYUIOPLKJHGFDSAZXCVBNM;
\def\bbc#1{\expandafter\def\csname bb#1\endcsname{{\mathbb #1}}}
\applytolist{bbc}QWERTYUIOPLKJHGFDSAZXCVBNM;
\def\bfc#1{\expandafter\def\csname bf#1\endcsname{{\mathbf #1}}}
\applytolist{bfc}QWERTYUIOPLKJHGFDSAZXCVBNM;
\def\sfc#1{\expandafter\def\csname s#1\endcsname{{\sf #1}}}
\applytolist{sfc}QWERTYUIOPLKJHGFDSAZXCVBNM;
\def\fc#1{\expandafter\def\csname f#1\endcsname{{\mathfrak #1}}}
\applytolist{fc}QWERTYUIOPLKJHGFDSAZXCVBNM;

\tikzstyle{shaded}=[fill=red!10!blue!20!gray!30!white]
\tikzstyle{unshaded}=[fill=white]
\tikzstyle{empty box}=[circle, draw, thick, fill=white, opaque, inner sep=2mm]
\tikzstyle{annular}=[scale=.7, inner sep=1mm, baseline]
\tikzstyle{rectangular}=[scale=.75, inner sep=1mm, baseline=-.1cm]
\usetikzlibrary{calc}

\newcommand{\roundNbox}[6]{
	\draw[rounded corners=5pt, very thick, #1] ($#2+(-#3,-#3)+(-#4,0)$) rectangle ($#2+(#3,#3)+(#5,0)$);
	\coordinate (ZZa) at ($#2+(-#4,0)$);
	\coordinate (ZZb) at ($#2+(#5,0)$);
	\node at ($1/2*(ZZa)+1/2*(ZZb)$) {#6};
}

\begin{document}
\title{
%$\rm C^*$-algebras from planar algebras III:
%\\
%{\large 
Compact quantum metric spaces from free graph algebras
%with an application to subfactor theory
%}
}
\author{Konrad Aguilar, Michael Hartglass, and David Penneys}
\date{\today}
\maketitle
\maketitle
\begin{abstract}
Starting with a vertex-weighted pointed graph $(\Gamma,\mu,v_0)$, we form the free loop algebra $\cS_0$ defined in 
Hartglass-Penneys' article on canonical $\rm C^*$-algebras associated to a planar algebra.
Under mild conditions, $\mathcal{S}_0$ is a non-nuclear simple $\rm C^*$-algebra with unique tracial state.
There is a canonical polynomial subalgebra $A\subset \mathcal{S}_0$ together with a Dirac number operator $N$ such that 
$(A, L^2A,N)$ is a spectral triple.
We prove the Haagerup-type bound of Ozawa-Rieffel to verify $(\mathcal{S}_0, A, N)$ yields a compact quantum metric space in the sense of Rieffel.

We give a weighted analog of Benjamini-Schramm convergence for vertex-weighted pointed  graphs.
As our $\rm C^*$-algebras are non-nuclear, we adjust the Lip-norm coming from $N$ to utilize the finite dimensional filtration of $A$.
We then prove that convergence of vertex-weighted pointed graphs leads to quantum Gromov-Hausdorff convergence of the associated adjusted compact quantum metric spaces.

As an application, 
we apply our construction to the Guionnet-Jones-Shyakhtenko (GJS) $\rm C^*$-algebra associated to a planar algebra.
We conclude that the compact quantum metric spaces coming from the GJS $\rm C^*$-algebras of many infinite families of planar algebras converge in quantum Gromov-Hausdorff distance.

%This is the submitted version of \arXiv{...}.
\end{abstract}

%%%%%%%%%%%%%%%%%%%%%%%%%%%%%%%%%%%%%%%%%%%%%%%%%%%%%%%%%%%
%%%%%%%%%%%%%%%%%%%%%%%%%%%%%%%%%%%%%%%%%%%%%%%%%%%%%%%%%%%
%%%%%%%%%%%%%%%%%%%%%%%%%%%%%%%%%%%%%%%%%%%%%%%%%%%%%%%%%%%
\section{Introduction}

%The $\rm C^*$-algebras we focus on are built from graphs that have a natural notion of convergence. 
%Thus, in this paper, we tackle the notion of whether the associated $\rm C^*$-algebras converge in Rieffel's quantum distance. 
%But, these $\rm C^*$-algebras are non-nuclear and finite-dimensional approximations are crucial in demonstrating convergence \cite{MR2055927, MR2520541,MR3718439,MR3778347,MR4266364}. In fact, we are unaware of continuity of quantum metric spaces built on non-nuclear $\rm C^*$-algebras in any quantum distance. 
%To approach this problem,  we show that these $\rm C^*$-algebras give rise to compact quantum metric spaces using canonical spectral triples and the results of \cite{MR2164594}. 
%Then, we further adjust these quantum metrics to take advantage of intrinsic finite-dimensional spaces.

In Connes' noncommutative geometry \cite{MR1007407,MR1303779}, the notion of a spectral triple is an analog of a space of smooth functions on a non-commutative manifold.
In \cite{MR1647515,MR1727499}, Rieffel initiated the study of noncommutative \emph{metric} geometry via the notion of a compact quantum metric space.
He then introduced \emph{quantum Gromov-Hausdorff distance} as a noncommutative analogue of Gromov-Hausdorff distance to provide a framework for establishing convergence of certain spaces arising in the operator algebra and high-energy physics literature \cite{MR2055927,MR2055928}.

To the best of our knowledge, all results proving convergence in quantum Gromov-Hausdorff distance do so for sequences of \emph{nuclear} $\rm C^*$-algebras, where finite-dimensional approximations are crucial in demonstrating convergence \cite{MR2055927, MR2520541,MR4015960,MR3718439,MR3778347,MR4266364}.
In this article, we prove a result about quantum Gromov-Hausdorff convergence for compact quantum metric spaces associated to \emph{non-nuclear} free graph algebras produced from vertex-weighted pointed graphs.

Given an unoriented connected graph $\Gamma=(V,E)$ with an arbitrary vertex weighting $\mu: V\to (0,\infty)$, 
we replace each edge $\varepsilon\in E$ between two distinct vertices by two oriented edges $\epsilon, \epsilon\op\in\vec{E}$ in opposite directions, and we replace each loop by a single oriented loop
to obtain a strongly connected directed graph $\vec{\Gamma}=(V, \vec{E})$ which inherits the same weighting $\mu$.
One forms the Toepltiz-Cuntz-Krieger graph algebra $\cT(\vec{\Gamma})$ \cite{MR1722197} with generators $\ell(\epsilon)$ for $\epsilon\in \vec{E}$, and the \emph{free graph algebra} \cite{MR3624399} is given by
$$
\cS(\Gamma,\mu)
=
{\rm C^*}
\left(
C_0(V) 
\cup 
\set{a_\epsilon \ell(\epsilon) + a_\epsilon^{-1}\ell(\epsilon\op)}{\epsilon \in \vec{E}} 
\right)
\subset \cT(\vec{\Gamma}).
$$
Here, each $a_\epsilon \in(0,\infty)$ depends on the weighting of the source and target of $\epsilon$, which is chosen so that $\cS(\Gamma,\mu)$ has a semifinite trace $\Tr$.
By \cite{MR3679687}, $\cS(\Gamma,\mu)$ is simple exactly when 
\begin{equation}
\label{eq:SimplicityCondition}    
\mu(v) < \sum_{\substack{\e \in \vec{E} \\ s(\e) = v}} \mu(t(\e));
\end{equation}
we assume this condition in the sequel.

Now there are canonical projections $p_v\in \cS(\Gamma,\mu)$ for the vertices $v\in V$, and by simplicity \cite{MR3624399,MR3266249,MR3679687}, each compression $p_v \cS(\Gamma,\mu)p_v$ is Morita equivalent to $\cS(\Gamma,\mu)$.
We thus consider \emph{pointed} weighted graphs, which are equipped with a basepoint $v_0\in V$ such that $\mu(v_0)=1$.
We consider the \emph{free loop algebra} $\cS_0=\cS_{0}(\Gamma,\mu):= p_{v_0} \cS(\Gamma,\mu)p_{v_0}$, which can be described as generated by loops on $\Gamma$ based at $v_0$.
Under condition \eqref{eq:SimplicityCondition}, $\cS_0(\Gamma,\mu)$ also has unique trace \cite{MR3679687}.

The loop algebra $\cS_0$ has a dense $*$-subalgebra $A$ of finite linear combinations of loops, which acts by bounded operators on $L^2(A, \tr_0) \cong L^2(\cS_0, \tr_0)$.
Moreover, $A$ is \emph{filtered} by finite dimensional $*$-closed subspaces $A_n$ of linear combinations of loops of length at most $n$, which satisfy $A_m\cdot A_n \subset A_{m+n}$ and $A_0 = \bbC1_A$.
In this situation, by \cite[Lemma 1.1]{MR2164594}, the formula
$$
N = \sum_{n\geq 0} n \operatorname{Proj}_{A_n \ominus A_{n-1}},
$$
defines a \emph{Dirac number operator} which has bounded commutator with elements of $A$.
Thus $(A,L^2A,N)$ is a \emph{spectral triple} in the sense of Connes \cite{MR1303779}.
We prove the Haagerup-type inequality of \cite[Theorem 1.2]{MR2164594}, which gives the following theorem.

\begin{thmalpha}
\label{thm:CompactQuantumMetricSpace}
The Dirac number operator $N$ induces a Lip-norm $L$ on $A$, making $(\cS_0,A,L)$ a compact quantum metric space in the sense of Rieffel \cite{MR2055927}.
\end{thmalpha}

Thus given a connected, vertex-weighted pointed graph $(\Gamma,\mu,v_0)$, we get a canonical compact quantum metric space.
Given a sequence of connected vertex-weighted graphs $(\Gamma_n,\mu_n)$, we say it converges \emph{locally uniformly} to a limit $(\Gamma,\nu)$ if essentially on every ball of radius $R$ about $v$, the graphs $\Gamma_n$ eventually coincide with $\Gamma$, and the weights converge pointwise.
This is a weighted analog of Benjamini-Schramm convergence \cite{MR1873300}.
(See Definition \ref{defn:LocalUniformConvergence} for the precise definition.)
With this definition in hand, we can ask whether the associated compact quantum metric spaces $(\cS_0(\Gamma_n,\mu_n),A(\Gamma_n,\mu_n),L_n)$ converge in quantum Gromov-Hausdorff distance to $(\cS_0(\Gamma,\mu),A(\Gamma,\mu),L)$.

Unfortunately, we were unable to solve this question due to two main problems.
First, projecting an element in $A_{n}$ onto $A_{n-1}$ can increase the operator norm, similar to how truncating a Fourier series can increase the sup norm.
Second, these algebras are non-nuclear, so we are missing the finite dimensional approximations which were essential to the results \cite{MR2055927, MR2520541,MR4015960,MR3718439,MR3778347,MR4266364}.

In analogous situations \cite{MR1727499, 1612.02404}, one replaces the Lip norm $L$ with another Lip norm $\cL$ produced by a Minkowski functional.
In our setup, we choose $\cL$ so that it agrees with $L$ on the spaces $A_n \ominus A_{n-1}$ of homogeneous loops, i.e., spans of loops of the same length $n$.
While this produces a less canonical compact quantum metric space, 
these adjusted quantum metrics take advantage of the intrinsic finite-dimensional spaces $A_n \ominus A_{n-1}$ of homogeneous loops, which replace the finite dimensional approximations in the nuclear setting.
In \S\ref{sec:lipnorm} below, we are able to prove that these compact quantum metric spaces converge in quantum Gromov-Hausdorff distance to the desired limit.

\begin{thmalpha}
\label{thm:Convergence}
If the sequence of vertex-weighted pointed graphs $(\Gamma_n,\mu_n,v_0^n)$ converges locally uniformly to $(\Gamma,\mu,v_0)$, then the induced compact quantum metric spaces $(\cS_0(\Gamma_n, \mu_n), A(\Gamma_n,\mu_n), \cL_n)$ converge in quantum Gromov-Hausdorff distance to $(\cS_0(\Gamma,\mu), A(\Gamma,\mu), \cL)$.
\end{thmalpha}

\subsection*{Application to subfactor theory}

The original motivation in our two articles  \cite{MR3624399, MR3266249} was to develop a connection between subfactor theory and $\rm C^*$-algebras with a view toward connections to Connes' non-commutative geometry \cite{MR1303779}.
The standard invariant of a finite index subfactor forms a shaded subfactor planar algebra \cite{math.QA/9909027}.
Here, we work with unshaded unitary factor planar algebras, which correspond to symmetrically self-dual bifinite bimodules over some factor \cite{MR3405915,MR4133163}.
The more categorically minded reader may choose to work directly with a unitary tensor category as in \cite{MR4139893}.

A special application of the setup in this article is when:
\begin{itemize}
\item
$\Gamma$ is the principal graph of an unshaded unitary factor planar algebra $\cP_\bullet$, 
\item
$\mu$ is a quantum dimension vertex-weighting which satisfies the Frobenius-Perron condition, and
\item
$v_0 = \star$, the distinguished vertex corresponding to the empty diagram/monoidal unit object.
\end{itemize}
In this case, the cutdown $\cS_0$ of $\cS(\Gamma, \mu)$ at $v_0=\star$ is isomorphic to the Guionnet-Jones-Shlyakhtenko (GJS) $\rm C^*$-algebra \cite{MR3624399, MR3266249}.
This algebra is the $\rm C^*$-completion of the graded algebra $\Gr_0$ arising from their diagrammatic reproof \cite{MR2732052} of Popa's celebrated subfactor reconstruction theorem \cite{MR1334479}.

%\nn{drop this paragraph?}
%Starting with a subfactor planar algebra, Guionnet-Jones-Shlyakhtenko (GJS) give tower of graded algebras $\Gr_k$ together with faithful tracial states $\tr_k$ such that $\Gr_k$ acts on $H_k=L^2(\Gr_k, \tr_k)$ by bounded operators.
%They used the von Neumann completions $\cM_k=\Gr_k''\cap B(H_k)$, which are interpolated free group factors \nn{}, to give a diagrammatic reproof of Popa's celebrated subfactor reconstruction theorem \nn{}.
%That is, $\cM_0\subset \cM_1$ is a subfactor with standard invariant $\cP_\bullet$.

When $\cP_\bullet = \cN\cC_\bullet$, the factor planar algebra of non-commuting polynomials on self-adjoint variables $X_1,\dots, X_n$, $\Gr_0$ is exactly the algebra of non-commutative polynomials, and $\cS_0$ is Voiculescu's reduced $\rm C^*$-algebra generated by $n$ free semi-circular elements.
Thus we may view $\Gr_0\subset \cS_0$ as a smooth subalgebra of polynomials inside the algebra of non-commuatative continuous functions.
We are thus in a position to study Connes' non-commutative geometry via Dirac operators and spectral triples \cite{MR1303779}.

In subfactor theory, there are many examples of local uniform graph convergence.
%$\Gamma_n=(V_n, E_n, \mu_n, \star_n) \to \Gamma=(V, E, \mu, \star)$.
For instance, we have examples coming from quantum groups at roots of unity \cite{MR0696688,MR936086,MR999799,MR1090432,MR1660937}, continuous families of subfactors \cite{MR2314611}, and from composites at a fixed index \cite{MR1386923,MR3345186}.
A particular application to subfactor theory is the following corollary of Theorem \ref{thm:Convergence}, which holds in much more generality than stated.

\begin{coralpha}
For a fixed $n$, the GJS $\rm C^*$-algebra $\cS_0(A_n,\exp(\pi i/(n+1)))$ of the Temperley-Lieb-Jones (TLJ) (sub)factor planar algebra gives a compact quantum metric space when equipped with the Lip norm $L_n$ from the Dirac number operator.
Adjusting our Lip norm to $\cL_n$ as in Theorem \ref{thm:Convergence}, the associated compact quantum metric spaces converge in quantum Gromov-Hausdorff distance to the adjusted compact quantum metric space of the GJS $\rm C^*$-algebra of TLJ at $q=1$.
\end{coralpha}

\subsection*{Acknowledgements.}
The authors would like to thank 
Farzad Fathizadeh,
Matilde Marcolli,
Marc Rieffel,
and
Robin Tucker-Drob
for helpful conversations.
David Penneys was supported by NSF DMS grants 1500387/1655912 and 1654159.

%%%%%%%%%%%%%%%%%%%%%%%%%%%%%%%%%%%%%%%%%%%%%%%%%%%%%%%%%%%
%%%%%%%%%%%%%%%%%%%%%%%%%%%%%%%%%%%%%%%%%%%%%%%%%%%%%%%%%%%
%%%%%%%%%%%%%%%%%%%%%%%%%%%%%%%%%%%%%%%%%%%%%%%%%%%%%%%%%%%
\section{Background}

%%%%%%%%%%%%%%%%%%%%%%%%%%%%%%%%%%%%%%%%%%%%%%%%%%%%%%%%%%%
\subsection{Compact quantum metric spaces}

We rapidly recall the notions of 
Gromov-Hausdorff distance, order unit space, compact quantum metric space, and quantum Gromov-Hausdorff distance from \cite{MR2055927}.

%We now recall all the notions of Hausdorff and Gromov-Hausdorff distance of compact metric spaces.
\begin{defn}
Suppose we have two compact subsets $X,Y$ of a metric space $(Z, \rho)$.
The Hausdorff distance between $X$ and $Y$ is given by
$$
\dist_{\text{H}}(X, Y) 
:= 
\inf\set{r>0}
{
X\subset N_r(Y) 
\text{ and } 
Y\subset N_r(X)
},
$$
where for $A\subset Z$, $N_r(A)$ is the $r$-neighborhood of $A$:
$$
N_r(A):=\set{z\in Z}{\text{there is an }a\in A\text{ with }\rho(z,a)<r}.
$$
\end{defn}

\begin{defn}
Now suppose $(X, \rho_X)$ and $(Y, \rho_Y)$ are independent compact metric spaces.
Let $X\amalg Y$ be the disjoint union of $X$ and $Y$,
and let $\cM(\rho_X,\rho_Y)$ be the set of all metrics $\rho$ on $X\amalg Y$ such that
\begin{itemize}
\item $\rho$ induces the disjoint union topology on $X\amalg Y$, and
\item $\rho|_X = \rho_X$ and $\rho|_Y = \rho_Y$.
\end{itemize}
The Gromov-Hausdorff distance between $(X, \rho_X)$ and $(Y, \rho_Y)$ is given by
$$
\dist_{\text{GH}}(X, Y) 
= 
\inf \set{\dist^\rho_{\text{H}}(X, Y)}
{
X, Y \subset (X\amalg Y, \rho) 
\text{ and }
\rho \in \cM(\rho_X, \rho_Y)
}.
$$
\end{defn}

\begin{defn}
An order unit space $(V, e)$ is a real vector space $V$ together with a partial order $\leq$ with an element $e$ called the order unit which satisfies
\begin{itemize}
\item
(order unit) 
For every $v\in V$, there is an $r>0$ such that $v\leq re$.
\item
(Archimedian property) 
If $v\leq re$ for all $r>0$, then $v\leq 0$.
\end{itemize}
An order unit space has a norm, which is given by
$\|v\| 
= 
\inf\set{r>0}{-r e \leq v \leq re}$.
\end{defn}

\begin{ex}
Suppose $\cA$ is a unital $\rm C^*$-algebra.
Then the self-adjoint elements $\cA_{\text{s.a.}}$ of $\cA$ form an order unit space with order unit $1_\cA$.
\end{ex}

\begin{defn}
Suppose $(V,e)$ is an order unit space.
\begin{itemize}
\item A state of $(V,e)$ is a continuous linear functional $\mu\in V^*$ such that $\mu(e) = 1 = \|\mu\|$.
The space of states on $(V,e)$ is denoted $S(V)$.
Given a seminorm $L$ on $(V,e)$, it induces a $[0,\infty]$-valued metric on $S(V)$ by
$$
\rho_L(\mu, \nu) = \sup\set{|\mu(v)-\nu(v)|}{L(v)\leq 1}.
$$
\item A Lip-norm on $(V,e)$ is a seminorm $L$ on $V$ such that
\begin{enumerate}[(1)]
\item
$L(v) = 0$ if and only if $v\in\bbR e$.
\item
The topology on $S(V)$ induced by $\rho_L$ is the weak-* topology.
\end{enumerate}
Note that (2) above implies $\rho_L$ is a genuine metric on $S(V)$ which takes only finite values.
\end{itemize}
\end{defn}

\begin{defn}
A compact quantum metric space is a triple $(V, e, L)$ where $(V,e)$ is an order unit space and $L$ is a Lip-norm on $(V,e)$.
\end{defn}

The following criterion will be useful in determining whether $L$ is a Lip-norm on a unital $\rm C^*$-algebra $\cA$.
A \emph{unital pre $\rm C^*$-algebra} is a pair $(A,\phi)$ where $A$ is a unital complex $*$-algebra and $\phi: A\to \bbC$ is a positive linear functional ($\phi(a^*a)\geq 0$) with $\phi(1_A)=1_\bbC$ such that the left action of $A$ on $L^2(A,\phi)$ is by bounded operators.

\begin{prop}[{\cite[Prop.~1.3]{MR2164594}}]
\label{prop:totallybounded}
Let $(A,\phi)$ be a unital pre $\rm C^*$-algebra, and let $L$ be a seminorm on $A$.  
Then $L$ is a Lip-norm if and only if
$$
\set{a \in A}{L(a) \leq 1 \text{ and } \phi(a) = 0}
$$
is a norm totally bounded subset of $A$.
\end{prop}

%Now suppose $\pi:(V, e_V) \to (W,e_W)$ is a surjective morphism of order unit spaces, i.e., $\pi$ is a positive linear map such which maps $e_V$ to $e_W$ (and thus has norm 1).
%Then $\pi$ induces an injection $\pi^* : S(W) \to S(V)$ by $\mu\mapsto \mu\circ \pi$.
%Given a Lip-norm $L_V$ on $(V, e_V)$, by \cite[Proposition 3.1]{MR2055927}, there is an induced Lip-norm on $(W,e_W)$ given by
%$$
%L_W(w) = \inf\set{L(v)}{\pi(v) = w}.
%$$

%Rieffel \nn{} gave a definition of quantum Gromov-Hausdorff distance for compact quantum metric spaces.

\begin{defn}[{\cite[Sections 3 and 4]{MR2055927}}]
Suppose we have compact quantum metric spaces $(V, e_V, L_V)$ and $(W, e_W, L_W)$.
Then $(V\oplus W, (e_V, e_W))$ is an order unit space.
Let $\cM(V, W)$ be the set of all Lip-norms $L$ on $V\oplus W$ which induce $L_V$ on $(V, e_V)$ and $L_W$ on $(W, e_W)$.
%(See \cite[Sections 3 and 4]{MR2055927} for more details.)
%This means that the Lip-norms on $(V, e_V)$ and $(W, e_W)$ induced by $L$ via the projection maps are equal to $L_V$ and $L_W$ respectively.
The quantum Gromov-Hausdroff distance between $(V, e_V, L_V)$ and $(W, e_W, L_W)$ is
$$
\dist_{q}(V, W) 
:= 
\inf \set{\dist_H^{\rho_L}(S(V), S(W))}{L\in \cM(L_V, L_W)}.
$$
\end{defn}

The following lemma  
%which is an immediate consequence of
%\cite[Thm.~6.3]{MR3413867} and the discussion preceding \cite[Def.~3.14]{MR3413867},
will help in providing important estimates later. 
%This is an immediate consequence of \cite[Theorem 6.3]{MR3413867} and the discussion preceding \cite[Definition 3.14]{MR3413867}. 

\begin{lem}
\label{lem:DistQFromDistH}
Let $(V,e,L)$ be a compact quantum metric space, and let $W \subseteq V$ be a unital subspace ($e\in W$) such that $(W,e,L|_W)$ is a compact quantum metric space.
If $\phi \in S(V)$, then 
% \begin{align*}
% & \dist_H^{\|\cdot\|}(\{a \in V \mid L_V(a) \leq 1\},\{a \in W \mid L_V(a) \leq 1\})\\
% & \quad \leq \dist_H^{\|\cdot\|}(\{a \in W \mid L_V(a) \leq 1 \text{ and } \phi(a)=0\},\{a \in W \mid L_V(a) \leq 1 \text{ and } \phi(a)=0\}).
% \end{align*}
% Moreover,
\[
\dist_q (V,W) \leq 2 \dist_H^{\|\cdot\|}(
\{a \in V \mid L(a) \leq 1 \text{ and } \phi(a)=0\},
\{a \in W \mid L|_W(a) \leq 1 \text{ and } \phi(a)=0\}
).
\]
\end{lem}
\begin{proof}
By \cite[Thm.~6.3]{MR3413867}, we have that
\[
\dist_q (V,W) \leq 2 \dist_H^{\|\cdot\|}(
\{a \in V \mid L(a) \leq 1\},
\{a \in W \mid L|_W(a) \leq 1 \}
).
\]
However, by the discussion preceding \cite[Def.~3.14]{MR3413867}, we have that
\begin{align*}
   & \dist_H^{\|\cdot\|}(
   \{a \in V \mid L(a)\leq 1 \},
   \{a \in W \mid L|_W(a) \leq 1 \}
   )\\
   & \quad \leq \dist_H^{\|\cdot\|}(
   \{a \in V \mid L(a) \leq 1 \text{ and } \phi(a)=0\},
   \{a \in W \mid L|_W(a) \leq 1 \text{ and } \phi(a)=0\}
   ),
\end{align*}
which completes the proof.
\end{proof}

%%%%%%%%%%%%%%%%%%%%%%%%%%%%%%%%%%%%%%%%%%%%%%%%%%%%%%%%%%%
\subsection{The Ozawa-Rieffel criterion}{\label{sec:OR}}

In \cite{MR2164594}, Ozawa and Rieffel give criteria to determine when a filtered $*$-algebra with a tracial state gives a compact quantum metric space.
We now recall their setup and theorem.

%Suppose $(A, \tr)$ is a complex $*$-algebra with a trace $\tr: A\to \bbC$.
%For this section, we place the following assumptions on $A$:
\begin{assumptions}
\label{assume:FilteredAlgebra}
For this section, $A$ is a unital complex $*$-algebra equipped with a trace $\tr: A\to \bbC$ such that $(A,\tr)$ is a pre $\rm C^*$-algebra.
We further assume:
\begin{itemize}

\item 
$A$ is filtered by $*$-closed finite dimensional subspaces.
That is, there are finite dimensional $*$-closed subspaces $A_0 \subseteq A_1 \subseteq A_2 \subseteq \cdots$ whose union is $A$ which satisfy $A_m \cdot A_n \subseteq A_{m+n}$.

\item
The ground algebra $A_0$ is trivial, i.e., $A_0 = \bbC 1_A$.

\item 
The left (and right) action(s) of $A$ on $(A, \tr)$ is bounded in $\|\cdot\|_2$, and thus extends to an action on $\cH = L^{2}(A, \tr)$ by bounded operators.

\end{itemize}
\end{assumptions}

Under these assumptions, we set $W_n = A_{n} \ominus A_{n-1}$ which is finite dimensional, and we let $P_{n}$ be the orthogonal projection from $\cH$ onto $W_n$.

\begin{defn}
The \emph{Dirac number operator} is defined by $N := \sum_{n \geq 0} nP_{n}$, which is closable with dense domain.
\end{defn}

One has the following lemma due to \cite{MR2164594}.

\begin{lem}[{\cite[Lemma 1.1]{MR2164594}}]
\label{lem:BoundedCommutator}
For every $a \in A$, $[N,a]$ is densely defined and extends to a bounded operator on $\cH$.
\end{lem}

Using this lemma, we define a seminorm $L$ on $A$ by $L(a) = \|[N, a]\|$.
Observe that $L$ vanishes exactly on $A_0=\bbC 1_A$.
We set $\cA = \overline{A}^{\|\cdot\|}$, and we denote by $S(\cA)$ the state space of $\cA$.
In the theory of compact quantum metric spaces, one induces a metric $\rho$ on $S(\cA)$ with values in $[0,\infty]$ by the formula
$$
\rho(\phi, \psi) = \sup\set{|\phi(x) - \psi(x)| \,}{ \, x\in A\text{ with }L(x) \leq 1}.
$$
If the topology induced by $\rho$ induces the weak-$*$ topology on $S(\cA)$, we say that $(\cA, A, L)$ is a \emph{compact quantum metric space} in the sense of Rieffel \cite{MR2055927}.  A main result of \cite{MR2164594} is the following theorem:

\begin{thm}[{\cite[Theorem 1.2]{MR2164594}}]
\label{thm:OzawaRieffel}
If there exists a $C>0$ such that for all $m, n, k \in \N$ and $a_{k} \in W_{k}$
$$
\|P_{m}a_{k}P_{n}\| \leq C\|a_{k}\|_{2},
$$
then $(\cA, A, L)$ is a compact quantum metric space.
\end{thm}

%Whenever a filtered $A$ as above satisfies this inequality, we say that $A$ satisfies a Haagerup-type inequality.

%%%%%%%%%%%%%%%%%%%%%%%%%%%%%%%%%%%%%%%%%%%%%%%%%%%%%%%%%%%
\subsection{Free graph algebras}

Let $\Gamma = (V, E)$ be a countable, connected, locally finite, undirected graph, 
and let $\mu : V \to \bbR_{>0}$ be a weighting on the vertices.
(The examples in the later part of this article will be principal graphs of 
%finite-depth 
(sub)factor planar algebras with a quantum dimension weighting which satisfies the Frobenius-Perron condition.)

We now follow the construction from \cite{MR3624399}, summarized in \cite[Section 2.1]{MR3679687}.

\begin{defn}
From our undirected graph $\Gamma$, we form a directed graph $\vec{\Gamma}=(V, \vec{E}, s, t)$ as follows.
\begin{enumerate}[(1)]
\item
For each $e \in E$ which has endpoints $\alpha\neq \beta$ in $V$,
we get two directed edges $\epsilon$ and $\epsilon\op$ in $\vec{E}$ such that
$$
s(\epsilon) = t(\epsilon\op) = \alpha 
\qquad \text{ and }\qquad
s(\epsilon\op) = t(\epsilon) = \beta. 
$$

\item
For each $e \in E$ which is a loop at the vertex $\gamma$, we get one directed edge $\epsilon \in \vec{E}$ with $s(\epsilon) = t(\epsilon) = \gamma$.

\end{enumerate}
Note that $\vec{\Gamma}$ inheirits the weighting $\mu$ from $\Gamma$.
\end{defn}

The algebra $C_0(V)$ is the $\rm C^*$-algebra generated by the indicator functions $p_v$ for $v\in V$ acting on $\ell^2(V)$.
The $C_0(V)$ Hilbert bimodule $\cX$ is the completion of the space of formal finite $\bbC$-linear combinations of edges of $\vec{E}$, under the $C_0(V)$-valued inner product given by $\langle \epsilon, \epsilon'\rangle = \delta_{\epsilon, \epsilon'} p_{t(\epsilon)}$.
The action of $C_0(V)$ on $\cX$ is given by $p_\alpha \epsilon = \delta_{\alpha, s(\epsilon)} \epsilon$ and $\epsilon p_\beta = \delta_{\alpha, t(\epsilon)} \epsilon$.

We then form the Pimsner-Fock space 
$$
\cF(\vec{\Gamma}) = C_0(V) \oplus \bigoplus_{n\geq 1} \bigotimes_{C_0(V)}^n \cX.
$$
The spaces $\bigotimes_{C_0(V)}^n \cX$ are spanned by elements of the form $\epsilon_1\otimes \cdots \otimes \epsilon_n$ such that $\epsilon_1\cdots \epsilon_n$ is a path in $\vec{\Gamma}$.
For each edge $\epsilon\in \vec{E}$, we get bounded creation and annihilation operators on $\cF(\vec{\Gamma})$ given by 
\begin{align*}
\ell(\epsilon) \left(\epsilon_1\otimes \cdots \otimes \epsilon_n \right) 
&:= 
\epsilon\otimes \epsilon_1 \otimes \cdots \otimes \epsilon_n
\\
\ell(\epsilon)^* \left(\epsilon_1\otimes \cdots \otimes \epsilon_n \right) 
&:= 
\langle\epsilon| \epsilon_1\rangle_{C_0(V)} \epsilon_2\otimes \cdots \otimes \epsilon_n.
\end{align*}
The Pimsner-Toeplitz algebra $\cT(\vec{\Gamma})$ is the $\rm C^*$-algebra generated by the $\ell(\epsilon), \ell(\epsilon)^*$.

\begin{defn}
The free graph algebra $\cS=\cS(\Gamma, \mu)$ is generated by the edge elements 
$$
X_{\e} = a_\e\ell(\e) + a_\e^{-1}\ell(\e\op)
\qquad \text{where} \qquad
a_\e = \sqrt[4]{\frac{\mu(s(\e))}{\mu(t(\e))}}.
$$ 
Note that $a_\epsilon^{-1} = a_{\epsilon\op}$.
We set $X_{e} := X_{\e} + X_{\e\op}$.  
\end{defn}

We now give the structure of the free graph algebra.
There is a conditional expectation $E: \cS(\Gamma, \mu) \rightarrow C_{0}(V)$ given by $ E(x) = \sum_{v \in V(\Gamma)} \langle v | xv \rangle_{C_{0}(V)} $.  
We have the following lemma.

\begin{lem}[{\cite{MR2807103, MR3624399}}]
The algebras $\cS_{e, \mu} = C^{*}(C_{0}(V), X_{e})$ are free with amalgamation over $C_{0}(V)$ with respect to the conditional expectation $E$, i.e.
$$
\cS(\Gamma, \mu) = \underset{C_{0}(V)}{\Asterisk} (\cS_{e, \mu}, E).
$$
Furthermore, $\mu \circ E$ defines a (semifinite) a trace $\Tr$ on $\cS$.
\end{lem}

One can check that $\Tr(p_{v}) = \mu(v)$ for all $v \in V$.

%%%%%%%%%%%%%%%%%%%%%%%%%%%%%%%%%%%%%%%%%%%%%%%%%%%%%%%%%%%
%%%%%%%%%%%%%%%%%%%%%%%%%%%%%%%%%%%%%%%%%%%%%%%%%%%%%%%%%%%
%%%%%%%%%%%%%%%%%%%%%%%%%%%%%%%%%%%%%%%%%%%%%%%%%%%%%%%%%%%
\section{Free loop algebras and compact quantum metric spaces}

%%%%%%%%%%%%%%%%%%%%%%%%%%%%%%%%%%%%%%%%%%%%%%%%%%%%%%%%%%%
\subsection{Free loop algebras give compact quantum metric spaces}
\label{sec:FreeLoopAlgebras}

Let $\cH_{\Gamma, \mu} = \cF(\vec{\Gamma}) \otimes_{C_{0}(V)} \ell^{2}(V(\Gamma), \mu)$. 
Here, $\ell^{2}(V(\Gamma), \mu)$ is the Hilbert space spanned by $V$, whose inner product is given by $\langle v, w \rangle = \delta_{v = w} \mu(v)^{2}$.   
Note that paths in $\vec{\Gamma}$ give an orthogonal basis for $\cH_{\Gamma, \mu}$,
and $\|\e_{1}\otimes\cdots \otimes \e_{n}\|_{\cH_\Gamma, \mu} = \sqrt{\mu(t(\e_{n}))}.$

We introduce the following notation.
\begin{nota}
Let $\Pi$ denote the set of all paths in $\vec{\Gamma}$, and $\sigma = \e_{1}\cdots\e_{n} \in \Pi$. Let $|\sigma|=n$ denote the length of $\sigma$.  
We set:
\begin{itemize}
\item 
$\sigma\op := \e_{n}\op \dots \e_{1}\op$

\item 
$\ell(\sigma) := \ell(\e_1)\cdots\ell(\e_n)$.

\item 
$a_{\sigma} := \sqrt[4]{\frac{\mu(s(\sigma))}{\mu(t(\sigma))}} = \sqrt[4]{\frac{\mu(s(\e_{1}))}{\mu(t(\e_{n}))}}$

\item 
$X_{\sigma} := X_{\e_1}\cdots X_{\e_n}$.

\item 
$Y_{\sigma} := \sum_{\sigma = \rho\tau} a_{\rho}a_{\tau}^{-1}\ell(\rho)\ell(\tau\op)^*$.
\end{itemize}

Once it is shown in Proposition \ref{prop:changeofbasis} below that $Y_{\sigma} \in \cS(\Gamma, \mu)$, it will follow from faithfulness of the trace that $Y_{\sigma}$ is the unique element in $\cS(\Gamma, \mu)$ whose right support is under $p_{t(\sigma)}$ and satisfies $Y_\sigma \cdot p_{t(\sigma)} = a_{\sigma}\sigma$.  
The element $Y_\sigma$ is known as the \emph{Wick word} of $a_{\sigma}\sigma$.
Observe that $Y_{\sigma}^*=Y_{\sigma^{\op}}$.
\end{nota}

We now perform a change of basis from the $X$'s to $Y$'s.
These $Y$'s will be useful later on as they are eigenvectors of the number operator.

\begin{prop}[Change of basis]\label{prop:changeofbasis}
Suppose $\sigma$ is a path in $\vec{\Gamma}$ of length $n$.
\begin{enumerate}[(1)]

\item 
$Y_{\sigma} \in \cS(\Gamma, \mu)$

\item
$Y_{\sigma} = X_{\sigma} + Q$ 
where $Q$ is a linear combination of the $X_{\sigma'}$ with $|\sigma'|<n$.

\item 
$X_{\sigma} = Y_{\sigma} + P$
where $P$ is a linear combination of the $Y_{\sigma'}$ with $|\sigma'|<n$.

\end{enumerate}
\end{prop}
\begin{proof}

We will prove this by induction on $|\sigma|$, the length of $\sigma$.  If $|\sigma| = 1$, then $\sigma = \e$ for some $\e \in \vec{E}$ and it is apparent that $Y_{\e} = X_{\e}$. 

Given, $\sigma$ with $|\sigma| > 1$, write $\sigma = \e\tau$ for $\e \in \vec{E}$, and write $\tau = \e'\tau'$ for $\e' \in \vec{E}$.  We see that

\begin{align*}
X_{\e}Y_{\tau} 
&= 
(a_{\e}\ell(\e) + a_{\e\op}\ell(\e\op)^{*})
\sum_{\tau = \tau_{1}\tau_{2}} 
a_{\tau_{1}}a_{\tau_{2}}^{-1}\ell(\tau_{1})\ell(\tau_{2}\op)^*
\\&= 
\sum_{\tau = \tau_{1}\tau_{2}} 
a_{\e\tau_{1}}a_{\tau_{2}}\ell(\e\tau_{1})\ell(\tau_{2}\op)^{*} 
+ 
\delta_{\e\op , \e'}
\sum_{\tau' = \tau'_{1}\tau'_{2}}
a_{\tau_{1}'}a_{(\tau'_{2})\op}^{-1}\ell(\tau_{1}')\ell((\tau'_{2})\op) 
+ 
a_{\e\op}a_{\tau}^{-1}\ell(\e\op)^{*}\ell(\tau^{op})^{*}
\\&= 
Y_{\sigma} 
+ 
\delta_{\e\op, \e'}Y_{\tau'}
\end{align*}
By induction, this proves (1) and (2).  
(3) follows directly from (2).
\end{proof}

We now work with the following assumption:

\begin{assumption}
\label{assume:MinimalWeight}
Our vertex-weighted graph $(\Gamma,\mu)$ comes with a pointing $v_0\in V$
with $\mu(v_0)=1$ which is minimal amongst vertex weights, i.e., $\mu(v_0)\leq \mu(v)$ for all $v\in V$.
\end{assumption}

\begin{defn}
From the free graph algebra $\cS=\cS(\Gamma, \mu)$, we define the free loop algebra at $v_{0}$, denoted $\cS_{0}$, as the cutdown at $p_{v_0}$, i.e., $\cS_{0} = p_{v_{0}}\cS p_{v_{0}}$.  
\end{defn}

\begin{rem}
In \cite{MR3624399, MR3679687} the K-theory of $\cS(\Gamma, \mu)$ was shown to be given by $K_{0}(\cS(\Gamma, \mu)) = \bbZ\set{[p_{v}]}{v \in V}$, the free abelian group generated by the equivalence classes of the projections $p_{v}$, and $K_{1}(\cS(\Gamma, \mu)) \cong \{0\}$.
Under the mild assumption \eqref{eq:SimplicityCondition}, $\cS(\Gamma,\mu)$ is simple, and thus $\cS_0(\Gamma,\mu)$ is Morita equivalent to $\cS(\Gamma,\mu)$.
Moreover, under condition \eqref{eq:SimplicityCondition}, $\cS_0(\Gamma,\mu)$ has unique trace.
It follows that if $(\Gamma, \mu, v_0)$ and $(\Gamma', \mu', v_0)$ are two pointed weighted graphs, and the additive groups generated by $\set{\mu(v)}{v \in V}$ and $\set{\mu(v')}{v' \in V'}$ do not agree, then $\cS_{0}(\Gamma, \mu) \not\cong \cS_{0}(\Gamma', \mu')$.  
\end{rem}

Let $\cH_{0}$ be the GNS Hilbert space of $\cS_0$ under the finite trace $\tr_0=\Tr|_{\cS_0}$.  
Let $\Pi_{0}$ denote the set of loops based at $v_{0}$.  
Note that $\Pi_{0}$ is an orthonormal basis for $\cH_{0} = L^{2}(\cS_{0}, \tr_0)$.  
Furthermore, if $\sigma\in \Pi_{0}$, then $Y_{\sigma}$ is the unique element in $\cS_{0}$ satisfying $Y_{\sigma}v_{0} = \sigma$ (note that $a_{\sigma}$ is necessarily 1 if $\sigma$ is a loop).  
This means that we have the following important fact:

\begin{fact}
The set $\set{Y_{\sigma}}{\sigma \in \Pi_{0}}$ is an orthonormal basis for $L^{2}(\cS_{0}, \tr_0)$.
\end{fact}

We define an unbounded operator $N$ in $\cH_{0}$ by the closure of the operator satisfying $N(\sigma) = n\sigma$ whenever $\sigma$ is a loop of length $n$.   
Let $A$ be the unital $*$-algebra generated by the elements $Y_\sigma$, which by Proposition \ref{prop:changeofbasis}, is also generated by the elements $X_\sigma$.
Notice that  under the identification of $\cH_{0}$ with $L^{2}(\cS_0, \tr_0)$, 
we may realize $N: A \rightarrow A$ as an unbounded operator satisfying $N(Y_{\sigma}) = nY_{\sigma}$ whenever $\sigma$ is a loop of length $n$.
Observe that due to cutting down $\cS$ by $p_{v_{0}}$, it follows that the null space of $N$ is precisely scalar multiples of $p_{0}$, the identity in $\cS_{0}$.   
%In this manner, $A$ becomes filtered in the obvious way, and 
Set $A_{n} := \spann\set{Y_{\sigma}\, }{\, |\sigma| \leq n}$, which is $*$-closed and finite dimensional.  
Furthermore, it is straightforward to see that $A_{n}\cdot A_{m} \subset A_{n+m}$, giving a $*$-filtration of $A$ by finite dimensional subspaces.
We are now in position to use the Ozawa-Rieffel framework as in \S\ref{sec:OR}. 

As above, we set $W_{n} := \spann\set{Y_{\sigma}}{\, |\sigma| = n} = A_{n} \ominus A_{n-1}$, and we define $P_{n}$ to be the orthogonal projection from $\cH_{0}$ onto $W_n$.

\begin{lem}
\label{lem:OzawaRieffelConstant}
If $x \in W_{k}$, then $\|P_{m} x P_{n}\| \leq \|x\|_{2}$.
\end{lem}

\begin{proof}
Write $x = \sum_{|\sigma| = k} b_{\sigma}Y_{\sigma}$.  Note that $\|x\|^{2}_{2} = \sum_{|\sigma| = k}|b_{\sigma}|^{2}$. We need to show that if $\xi \in P_{n}\cH_{0}$ then we have
$$
\|P_{m}x\xi\|_{\cH_{0}} \leq \|x\|_{2}\|\xi\|_{\cH_{0}}
$$
Write $\xi = \sum_{|\tau| = n} c_{\tau}\tau$.  
The term $P_{m} x P_{n}$ is zero unless $|m-n| \leq k \leq m+n$.  
Choose $j$ such that $(k-j) + (n-j) = m$, and write
$$
x 
= 
\sum_{\substack{|\rho_{1}| = k-j\\ |\tau_{1}| = j}} b_{\rho_{1}\tau_{1}}Y_{\rho_{1}\tau_{1}} 
\qquad
\text{and} 
\qquad
\xi 
= 
\sum_{\substack{|\rho_{2}| = j\\ |\tau_{1}| = n-j}} 
c_{\rho_{2}\tau_{2}}\rho_{2}\tau_{2}.
$$
This means that
\begin{align*}
P_{m}x\xi 
&= 
\left(
\sum_{\substack{|\rho_{1}| = k-j\\ |\tau_{1}| = j}} b_{\rho_{1}\tau_{1}}a_{\rho_{1}}^{2}\ell(\rho_{1})\ell(\tau_{1}\op)^{*}
\right)
\cdot 
\sum_{\substack{|\rho_{2}| = j\\ |\tau_{1}| = n-j}} c_{\rho_{2}\tau_{2}}\rho_{2}\tau_{2}
%\\&
= 
\sum_{\substack{|\rho| = k-j\\ |\sigma| = j \\ |\tau| = n-j}} a_{\sigma}^{2}b_{\rho \sigma}c_{\sigma\tau}\rho\tau.
\end{align*}
From this, we see that
\begin{align*} 
\|P_{m}x\xi\|^{2}_{2} &= \sum_{\substack{|\rho| = k-j\\ |\tau| = n-j}} \left|\sum_{|\sigma| = j}a_{\sigma}^{2}b_{\rho \sigma}c_{\sigma\tau}\right|^{2}\\
&\leq 
\sum_{\substack{|\rho| = k-j\\ |\tau| = n-j}} 
\left(\sum_{|\sigma| = j}|a_{\sigma}^{2}b_{\rho \sigma}c_{\sigma\tau}|\right)^{2}
\\&\leq 
\sum_{\substack{|\rho| = k-j\\ |\tau| = n-j}} 
\left(\sum_{|\sigma| = j}|b_{\rho \sigma}c_{\sigma\tau}|\right)^{2}
\\&\leq 
\sum_{\substack{|\rho| = k-j\\ |\tau| = n-j}} 
\left[
\left(\sum_{|\sigma_{1}| = j}|b_{\rho\sigma_{1}}|^{2} \right) 
\left(\sum_{|\sigma_{2}| = j}|c_{\sigma_{1}\tau}|^{2} \right)  
\right]
\\&\leq 
\left( \sum_{|\sigma'| = k} |b_{\sigma'}|^{2} \right)
\cdot 
\left( \sum_{|\sigma''| = n} |c_{\sigma''}|^{2} \right)
\\&= 
\|x\|_{2}^{2}\cdot\|\xi\|_{2}^{2}
\end{align*}
as desired.
\end{proof}

Lemma \ref{lem:OzawaRieffelConstant} immediately implies Theorem \ref{thm:CompactQuantumMetricSpace}.

\begin{proof}[Proof of Theorem \ref{thm:CompactQuantumMetricSpace}]
Lemma \ref{lem:OzawaRieffelConstant} allows us to use the Ozawa-Rieffel criterion in Theorem \ref{thm:OzawaRieffel} (for $C = 1$) to conclude that $(\cS_0, A, N)$ is a compact quantum metric space.
\end{proof}

%%%%%%%%%%%%%%%%%%%%%%%%%%%%%%%%%%%%%%%%%%%%%%%%%%%%%%%%%%%
\subsection{Convergence for weighted pointed graphs}

We now discuss a type of convergence for vertex-weighted pointed graphs, which is a weighted analog of Benjamini-Schramm convergence \cite{MR1873300}.
As in the previous sections, the graphs we consider are countable, connected, locally finite, undirected, vertex-weighted, and pointed, where the base-point has minimal weight $1$.
The following notation will be handy.

\begin{nota}
Suppose $\Gamma = (V, E, v)$ is such a graph and $R\in\bbN$.
We denote by $\Gamma(R)$ the truncation of $\Gamma$ to the ball of radius $R$ of $\Gamma$ based at $v$.
\end{nota}

\begin{defn}
\label{defn:LocalUniformConvergence}
Suppose we have a sequence of such graphs $(\Gamma_n = (V_n, E_n, v_n),\mu_n)$, and another graph $(\Gamma=(V, E, v),\mu)$.
We say that $\Gamma_n$ converges locally to $\Gamma$ if for all $R\in \bbN$, there is an $N_R>0$ such that for every $n>N_R$, 
\begin{itemize}
\item
there is a pointed graph isomorphism $\varphi_n^R : \Gamma_n(R) \to \Gamma(R)$, and
\item
these graph isomorphisms satisfy for all $n>\max\{N_R, N_{R+1}\}$, $\varphi_n^{R+1}|_{\Gamma_n(R)} = \varphi_n^R$.
\end{itemize}
Moreover, we say $\Gamma_n \to \Gamma$ locally uniformly if $\Gamma_n \to \Gamma$ locally and the isomorphisms $\varphi_n^R$ satisfy 
\begin{itemize}
\item
for every vertex $w\in \Gamma$ with $\operatorname{dist}(v, w)\leq R$, $\lim_{n\to \infty} \mu_n[(\varphi_n^R)^{-1}(w)] = \mu(w)$.
\end{itemize}
\end{defn}

\begin{exs}
\label{exs:GraphConvergence}
We give several examples of local uniform graph convergence.
\begin{enumerate}[(1)]
\item
(Subgraphs converging to a limit graph)
Consider the Coxeter-Dynkin diagrams $A_n$ with their unique normalized Frobenius-Perron weighting, where the base-point is at the left:
$$
\begin{tikzpicture}[baseline=-.1cm]
	\draw (0,0) -- (2.5,0);
	\draw (3.5,0) -- (5,0);
	\filldraw (0,0) circle (.05cm) node [above] {\scriptsize{$[1]$}} node [below] {\scriptsize{$\star$}};
	\filldraw (1,0) circle (.05cm) node [above] {\scriptsize{$[2]$}};
	\filldraw (2,0) circle (.05cm) node [above] {\scriptsize{$[3]$}};
	\node at (3.03,0) {$\cdots$};
	\filldraw (4,0) circle (.05cm) node [above] {\scriptsize{$[n-1]$}};
	\filldraw (5,0) circle (.05cm) node [above] {\scriptsize{$[n]$}};
\end{tikzpicture}
\qquad
\text{where}
\qquad
[n] = \frac{q^n - q^{-n}}{q-q^{-1}}
\qquad
\text{and}
\qquad
q=\exp\left(
\frac{2\pi i}{2n}
\right).
$$
It is easy to see that the $A_n$ converge to the Coxeter-Dynkin diagram $A_\infty$ with its Frobenius-Perron weighting
$$
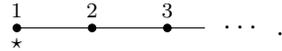
\begin{tikzpicture}[baseline=-.1cm]
	\draw (0,0) -- (2.5,0);
	\filldraw (0,0) circle (.05cm) node [above] {\scriptsize{$1$}} node [below] {\scriptsize{$\star$}};
	\filldraw (1,0) circle (.05cm) node [above] {\scriptsize{$2$}};
	\filldraw (2,0) circle (.05cm) node [above] {\scriptsize{$3$}};
	\node at (3,0) {$\cdots$};
\end{tikzpicture}
\,.
$$
Just observe that as $\theta \to 0$, we have $q=e^{i\theta}\to 1$, so $[n]\to n$.

\item
(Weightings converging on the same graph)
We fix the graph $A_\infty$, but we consider the continuous family of Frobenius-Perron weightings given by
$$
\begin{tikzpicture}[baseline=-.1cm]
	\draw (0,0) -- (2.5,0);
	\filldraw (0,0) circle (.05cm) node [above] {\scriptsize{$[1]$}} node [below] {\scriptsize{$\star$}};
	\filldraw (1,0) circle (.05cm) node [above] {\scriptsize{$[2]$}};
	\filldraw (2,0) circle (.05cm) node [above] {\scriptsize{$[3]$}};
	\node at (3,0) {$\cdots$};
\end{tikzpicture}
\qquad
\text{where}
\qquad
[n] = \frac{q^n - q^{-n}}{q-q^{-1}}
\qquad
\text{and}
\qquad
q\geq 1.
$$
It is easily verified that any convergent sequence $q_n \to q_0$ gives a convergent sequence of graphs.
\item
(Existence of only local isomorphisms)
Consider the the affine Coxeter-Dynkin diagrams $D_n^{(1)}$ with their unique normalized Frobenius-Perron weighting:
$$
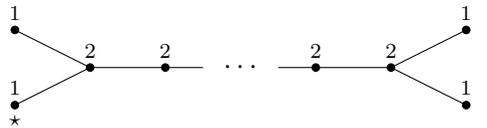
\begin{tikzpicture}[baseline=-.1cm]
	\draw (0,-.5) -- (1,0) -- (0,.5);
	\draw (6,-.5) -- (5,0) -- (6,.5);
	\draw (1,0) -- (2.5,0);
	\draw (3.5,0) -- (5,0);
	\filldraw (0,-.5) circle (.05cm) node [above] {\scriptsize{$1$}} node [below] {\scriptsize{$\star$}};
	\filldraw (0,.5) circle (.05cm) node [above] {\scriptsize{$1$}};
	\filldraw (1,0) circle (.05cm) node [above] {\scriptsize{$2$}};
	\filldraw (2,0) circle (.05cm) node [above] {\scriptsize{$2$}};
	\node at (3.03,0) {$\cdots$};
	\filldraw (4,0) circle (.05cm) node [above] {\scriptsize{$2$}};
	\filldraw (5,0) circle (.05cm) node [above] {\scriptsize{$2$}};
	\filldraw (6,.5) circle (.05cm) node [above] {\scriptsize{$1$}};
	\filldraw (6,-.5) circle (.05cm) node [above] {\scriptsize{$1$}};
\end{tikzpicture}
\,.
$$
It is easily verified that these graphs converge to the affine Coxeter-Dynkin diagram $D_\infty$ with its Frobenius-Perron weighting
$$
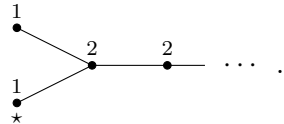
\begin{tikzpicture}[baseline=-.1cm]
	\draw (0,-.5) -- (1,0) -- (0,.5);
	\draw (1,0) -- (2.5,0);
	\filldraw (0,-.5) circle (.05cm) node [above] {\scriptsize{$1$}} node [below] {\scriptsize{$\star$}};
	\filldraw (0,.5) circle (.05cm) node [above] {\scriptsize{$1$}};
	\filldraw (1,0) circle (.05cm) node [above] {\scriptsize{$2$}};
	\filldraw (2,0) circle (.05cm) node [above] {\scriptsize{$2$}};
	\node at (3,0) {$\cdots$};
\end{tikzpicture}
\,.
$$
\end{enumerate}
\end{exs}

%%%%%%%%%%%%%%%%%%%%%%%%%%%%%%%%%%%%%%%%%%%%%%%%%%%%%%%%%%%
\subsection{Adjusting the Lip-norm}\label{sec:lipnorm}

Nuclearity is often used to establish  quantum Gromov-Hausdorff convergence of infinite-dimensional quantum metric spaces (see \cite{MR2055927, MR2520541,MR3718439,MR4015960,MR3778347,MR4266364} where nuclearity is either implicitly or explicitly used to provide finite-dimensional approximations). 
Since $\cS_0$ is exact but non-nuclear, we do not have contractive completely positive maps for finite dimensional approximations.
Instead, we pass to a new Lip-norm $\cL$ on $\cS_{0}$
defined from the finite dimensional spaces $A_n\ominus A_{n-1}$ from our filtration $(A_n)$ of $A\subset\cS_0$.
The spaces $A_n\ominus A_{n-1}$ of homogeneous loops provide an appropriate finite dimensional approximation.

\begin{defn}
Let $W_n:= \spann\set{Y_\sigma}{\,|\sigma|=n}$ i.e., the span of the Wick words of length $n$ in $\cS_0$, and observe $A_n = \spann \bigcup_{k=0}^n W_k$.
%We define $\cL$ on $A$ as follows:  Whenever $x \in B_{n}$, we set
%$$
%\cL(x) := L(x) = \|[N, x]\|
%$$
Set $B_{n} := \set{x \in W_{n}}{ L(x) \leq 1}$, and define
\begin{align*}
%$$
\cC' &:= \overline{\conv{\bigcup_{k=0}^{\infty} B_{k}}}, 
&
%$$
%and
%\begin{align*}
\cC
&:=
\overline{\conv{\bigcup_{k=1}^{\infty} B_{k}}},
&&\text{and}&
\cC_n
&:=
\conv{\bigcup_{k=1}^{n} B_{k}} 
= 
\overline{\conv{\bigcup_{k=1}^{n} B_{k}}}.
\end{align*}
We then define $\cL$ on $\cS_{0}$ to be the Minkowski functional associated to $\cC'$, i.e. 
\begin{equation}
\label{eq:ReplaceLipNorm}
\cL(x) := \inf\set{r > 0}{r^{-1}x \in \cC'}.
\end{equation}
Observe that $\cL(x)=\infty$ whenever $r^{-1}x\notin \cC'$ for all $r>0$.
Clearly $\cL$ is finite on $A$.
\end{defn}

\begin{rems}
\mbox{}
\begin{enumerate}[(1)]
\item
By construction, $L\leq \cL$.
\item 
By lower semi-continuity of $L$, we have that 
\[
\set{x \in \cS_0 }{ \cL(x) \leq 1} \subseteq \set{x \in \cS_0 }{ L(x) \leq 1}. 
\]
Thus, as $L$ is a Lip-norm, it follows from Proposition \ref{prop:totallybounded} that $\cL$ is a Lip-norm.
\item
Observe that if $x \in W_{n}$, then $L(x) = \cL(x)$.  
Indeed, by lower-semicontinuity of $L$, we see that $L\leq 1$ on $\cC$.
So if $L(x)=1$, then $\alpha x\notin \cC$ for all $\alpha >1$, and thus $\cL(x)=1$.
\end{enumerate}
\end{rems}

\begin{lem}
\label{lem:TailBound}
For every $\varepsilon>0$, there is a $K\in \bbN$ such that 
$x \in A \ominus A_{K}$ and $L(x) \leq 1$ 
implies
$\|x\|< 2\epsilon/3$.
Moreover,
for all $n>m >K$,
$\dist_{H}^{\|\cdot\|}(\cC_m, \cC_n) < 4\varepsilon/3$.
This result holds independent of the graph $\Gamma$ and the vertex-weighting $\mu$ on $\cS_0$.
\end{lem}
\begin{proof}
By \cite[\S3]{MR2164594}, given $a \in A$ with $L(a) \leq 1$ and $\varepsilon > 0$, there are $K > M >0$ depending only on the constant $C$ in Theorem \ref{thm:OzawaRieffel} (\cite[Theorem 1.2]{MR2164594}) so that
\begin{itemize}
\item
if $a^{(M)}=\sum_{|m-n|\geq M}P_m aP_n$ and $a^M = a-a^{(M)}$, then $\|a^{(M)}\|<\varepsilon/3$
\item
if $\widehat{a}_K = \sum_{k\leq K} a_k$ and $\widetilde{a}_K = a-\widehat{a}_K$ (which are both in $A(t)$), 
$\|(\widetilde{a}_K)^M\|<\varepsilon/3$.
\end{itemize}
Thus for $x \in A \ominus A_{K}$, we have
$$
\|x\| 
= 
\|(\widehat{x}_{K})^{M} + (\tilde{x}_{K})^{M} + x^{(M)}\| 
= 
\|(\tilde{x}_{K})^{M} + x^{(M)}\| < 2\varepsilon/3.
$$
It was shown in Lemma \ref{lem:OzawaRieffelConstant} that $C = 1$ regardless of the graph $\Gamma$ and the vertex-weighting $\mu$ (as long as the base point $v_0$ has minimal weighting).
Therefore, it follows that if $x \in A \ominus A_{K}$ and $\cL(x) \leq 1$, then $L(x) \leq 1$ and hence $\|x\|< 2\varepsilon/3$.

Now by definition, $\cC_m\subset \cC_n$.
Observe that when 
$V$ is a vector space and $S\subset T \subset V$,
then every element in $\conv(T)$ is a convex combination of an element in $\conv(S)$ and an element in $\conv(T\setminus S)$.
Hence if $x\in \cC_n$, we have $x=\lambda y+ (1-\lambda) z$
where $y\in \cC_m$ and $z \in \conv\bigcup_{k=m+1}^n B_k$.   
Since $z\in A\ominus A_K$ and $\cL(z) \leq 1$, $\|z\|<2\varepsilon/3$ by the preceding paragraph.
We conclude that $ \dist_H^{\|\cdot\|}(\cC_m, \cC_n) < 4\varepsilon/3$ as desired.
\end{proof}

\begin{cor}
\label{cor:TailBound}
For every $\varepsilon>0$, there is a $K\in \bbN$ such that $\dist_q((\cS_0,A,\cL),(A_k, \cL|_{A_k}))<\varepsilon$ for all $k \geq K.$
Again, this result holds independent of the graph $\Gamma$ and the vertex-weighting $\mu$ on $\cS_0$.
\end{cor}
\begin{proof}
Immediate from Lemmas \ref{lem:DistQFromDistH} and \ref{lem:TailBound}.
\end{proof}

\begin{lem}
\label{lem:ConvexHausdorffConvergence}
Let $(M,\rho)$ be a normed linear space.
Suppose we have $k>0$ sequences $(X^i_n)$, $i=1,\dots k$, of compact convex subsets of $Z$
and compact convex subsets $X^i \subset Z$
such that $X^i_n \to X^i$ in Hausdorff distance for each $i=1,\dots, k$.
Then
$$
\conv\left(\bigcup_{i=1}^k X^i_n\right)
\longrightarrow
\conv\left(\bigcup_{i=1}^k X^i\right)
$$
in Hausdorff distance.
\end{lem}
\begin{proof}
Let $\varepsilon>0$.
Choose $N$ sufficiently large such that $\dist_H(X^i_n, X^i)<\varepsilon$ for all $n>N$ and $i=1,\dots, k$.
Then for $x^i \in X^i$, $i=1,\dots, k$, we can choose $x^i_n\in X^i_n$ for $n>N$ with $\dist_H(x^i_n, x^i)<\varepsilon$.
Hence
$$
\left\|
\sum_{i=1}^k \lambda_i x^i_n 
-
\sum_{i=1}^k \lambda_i x^i 
\right\|
\leq
\sum_{i=1}^k \lambda_i \|x^i_n - x^i\|
<\varepsilon.
$$
This means every point in $\conv\left(\bigcup_{i=1}^k X^i\right)$ can be $\varepsilon$-approximated by a point in $\conv\left(\bigcup_{i=1}^k X^i_n\right)$ where $n>N$.
But the above argument clearly works swapping the roles of $\conv\left(\bigcup_{i=1}^k X^i\right)$ and $\conv\left(\bigcup_{i=1}^k X^i_n\right)$ for $n>N$.
The result follows.
\end{proof}

\begin{thm}
Let $(\Gamma_{n}, \mu_{n})$ be a sequence of vertex-weighted pointed graphs converging locally uniformly to $(\Gamma, \mu)$, each of which have base point $v_0$
(here, we suppress the isomorphism data $\varphi_n^R$).
Let $\cL_{n}$ be the Lip-norm constructed on $\cS_{0}(\Gamma_{n}, \mu_{n})$ as in \eqref{eq:ReplaceLipNorm} above.  
Then $(\cS_{0}(\Gamma_{n}, \mu_{n}), A(\Gamma_{n}, \mu_{n}), \cL_{n})$ converges in quantum Gromov-Hausdorff distance to $(\cS_{0}(\Gamma, \mu), A(\Gamma, \mu), \cL).$ 
\end{thm}
\begin{proof}
For any fixed $K \in \bbN$, $A_{K}(\Gamma_{n}, \mu_{n})$ is finite-dimensional.   
For sufficiently large $n$, all of the $\Gamma_{n}(K)$ coincide (again, we suppress the isomorphism data $\varphi_n^R$).  
It follows that for these values of $n$, the vector spaces $A_{K}(\Gamma_{n}, \mu_{n})$ are canonically isomorphic to the complex linear span of all loops in $\Gamma$ of length at most $K$ based at $v_0$.  
Setting 
$$
V_K
:=
\operatorname{span}_{\bbR}\set{a\sigma+\overline{a}\sigma^{\op}}{a\in \bbC\text{ and }\sigma\in \Pi_0\text{ such that }|\sigma|\leq K},
$$
$V_K$ is canonically isomorphic to $A_{K}(\Gamma, \mu)^{\rm sa}$ and $A_{K}(\Gamma_{n}, \mu_{n})^{\rm sa}$ for sufficiently large $n$.

To prove quantum Gromov-Hausdorff convergence, we use \cite[Theorem 11.2]{MR2055927} on a continuous field of order unit spaces whose underlying vector space is $V_K$ obtained by transporting the norms from the compact quantum metric spaces $(\cS_{0}(\Gamma_{n}, \mu_{n}), A(\Gamma_n,\mu_n), \cL_n)$.
Fix a loop $\sigma$ of $\Gamma$ based at $v_0$, and let $Y_{\sigma}(n)$ be the Wick-word for $\sigma$ in $\cS_{0}(\Gamma_{n}, \mu_{n})$:
$$
Y_{\sigma}(n) = \sum_{\sigma = \rho\tau} a_{\rho}(n)a_{\tau}^{-1}(n)\ell(\rho)\ell(\tau\op)^*.
$$
since $\mu_{n} \rightarrow \mu$ as $n \rightarrow \infty$, it follows that $a_{\rho}(n) \rightarrow a_{\rho}$ for any path $\rho$ in $\Gamma$.  
For $\xi\in V_K$, we write $Y_\xi$ for the corresponding linear combination of Wick words in $\cS_{0}(\Gamma, \mu)$,
and we write $Y_\xi(n)$ for the corresponding linear combination of Wick words in $\cS_{0}(\Gamma_{n}, \mu_{n})$ for sufficiently large $n$.
Setting $\|\xi\|_{n} := \|Y_{\xi}(n)\|_{\cS_{0}(\Gamma_{n}, \mu_{n})}$ and $\|\xi\| := \|Y_{\xi}\|_{\cS_{0}(\Gamma, \mu)}$,
we have $\|\xi\|_{n} \rightarrow \|\xi\|$ as $n \rightarrow \infty$.
Moreover, for each $k$ between $1$ and $K$, note that for any $\xi\in V_K$ which is a linear combination of loops of length exactly $k$, 
$\cL_{n}(Y_\xi(n))$ converges to $\cL(Y_{\xi})$.  
This is due to the fact that on the space $W_{k}$, $L$ and $\cL$ coincide. 

Now by \cite[Theorem 11.2]{MR2055927} and Lemma  \ref{lem:ConvexHausdorffConvergence}, we have that
\[
\lim_{n \to \infty} \dist_q ((A_{K}(\Gamma_{n}, \mu_{n})^{\rm sa}, \cL|_{A_{K}(\Gamma_{n}, \mu_{n})^{\rm sa}}),(A_{K}(\Gamma, \mu)^{\rm sa}, \cL|_{A_{K}(\Gamma, \mu)^{\rm sa}}))=0.
\]
Finally, using Corollary \ref{cor:TailBound}
and a standard $\varepsilon/3$ argument, the result follows.
\end{proof}
\section{Application to subfactor theory}

We refer the reader to \cite{MR2972458} for the definition of a subfactor planar algebra and its principal graphs and to \cite{MR3405915} for the definition of a factor planar algebra and its principal graph.

%%%%%%%%%%%%%%%%%%%%%%%%%%%%%%%%%%%%%%%%%%%%%%%%%%%%%%%%%%%
\subsection{The Guionnet-Jones-Shlyakhtenko \texorpdfstring{$\rm C^*$}{C*}-algebras}

Let $\cP_\bullet$ be a (sub)factor planar algebra. 
We now give the construction from \cite{MR3624399, MR3266249} of the Guionnet-Jones-Shlyakhtenko (GJS) $\rm C^*$-algebras based on the constructions \cite{MR2732052,MR2645882,MR2807103}.
A similar construction starting from a unitary tensor category and chosed symmetrically self-dual generator was given in \cite{MR4139893}.

First, we form the graded algebra $\Gr_0 = \bigoplus_{n\geq 0} \cP_n$ with the Bacher-Walker product
$$
x\star y
=
\sum_{j=0}^{\min\{m,n\}}
\begin{tikzpicture}[baseline=.2cm]
    \draw (-.15,0) -- (-.15,.7);
    \draw (.95,0) -- (.95,.7);
    \draw (.15,.3) arc (180:0:.25cm);
    \roundNbox{unshaded}{(0,0)}{.3}{0}{0}{$x$}
    \roundNbox{unshaded}{(.8,0)}{.3}{0}{0}{$y$}
    \node at (-.6,.5) {\scriptsize{$m-j$}};
    \node at (1.35,.5) {\scriptsize{$n-j$}};
    \node at (.4,.75) {\scriptsize{$j$}};
\end{tikzpicture}
\qquad
\text{for}
\qquad
x\in \cP_m, y\in \cP_n,
$$
and trace given by
$$
\tr(x) = \delta_{m=0} x
\qquad
\text{for}
\qquad
x\in \cP_m.
$$
(Since $\cP_0 = \bbC$, the above expression gives us a number.)
We note that under the GNS inner product $\langle x, y\rangle = \tr_0(y^*\star x)$, the spaces $\cP_n$ are orthogonal for distinct $n$.

Observe that since each $\cP_n$ is $*$-closed and finite dimensional, the subspaces $A_n = \bigoplus_{k=0}^n \cP_n$ give $\Gr_0$ the structure of a $*$-filtration by finite dimensional subspaces.
Moreover, since $\cP_\bullet$ is connected, $\cP_0 = \bbC 1_{\Gr_0}$.
Finally, by \cite{MR2732052,MR2645882}, the action of $\Gr_0$ on $(\Gr_0, \tr_0)$ is bounded in $\|\,\cdot \,\|_2$.
Hence Assumptions \ref{assume:FilteredAlgebra} hold, and we are in the position to use the Ozawa-Rieffel criterion from Theorem \ref{thm:OzawaRieffel} (\cite[Theorem 1.2]{MR2164594}).

\begin{defn}
The $\rm C^*$-algebra $\cA_0 = \overline{\Gr_0}^{\|\cdot\|}$ acting on $L^2(\Gr_0, \tr_0)$ is called the GJS $\rm C^*$-algebra of $\cP_\bullet$.
\end{defn}

Let $(\Gamma,\mu)$ be the principal graph of $\cP_\bullet$ with its quantum dimension weighting, which satisfies the Frobenius-Perron condition.
Note that the distinguished vertex $\star$ has minimal weight 1, so Assumption \ref{assume:MinimalWeight} holds.
We have the following lemma from \cite{MR3266249} which connects the GJS $\rm C^*$-algebra to the free loop algebras discussed in Section \ref{sec:FreeLoopAlgebras}.

\begin{lem}[{\cite[Cor.~3.4]{MR3266249}}]
The $\rm C^*$-algebra $\cA_0$ is isomorphic to the free loop algebra $\cS_0(\Gamma, \mu)$.
\end{lem}

\begin{rem}
The examples of local uniform graph congergence in Examples \ref{exs:GraphConvergence} are all examples of principal graphs of subfactors with weightings given by quantum dimension functions which satify the Frobenius-Perron condition.
We may thus interpret Theorem \ref{thm:Convergence} as giving quantum Gromov-Hausdorff convergence of the compact quantum metric spaces associated to GJS $\rm C^*$-algebras.
\end{rem}

%%%%%%%%%%%%%%%%%%%%%%%%%%%%%%%%%%%%%%%%%%%%%%%%%%%%%%%%
\subsection{The number operator}

As in \S\ref{sec:OR}, we have the number operator $N= \sum_{n\geq 0} nP_n$ on $\Gr_0$, where $P_n$ is the projection with range $B_n = A_n \ominus A_{n-1} = \cP_n$.
We end our article with some further observations about the properties of the number operator in our setup.

To begin, we give a supplementary diagrammatic proof that the number operator has bounded commutator with $\Gr_0$, although it follows directly from Lemma \ref{lem:BoundedCommutator} (\cite[Lemma 1.1]{MR2164594}).

\begin{lem}
\label{lem:DiagrammaticBoundedCommutator}
The number operator $N$ has bounded commutator with every $x\in \Gr_0$.
\end{lem}
\begin{proof}
To show $\| [N, x]\|_\I$ is bounded for an arbitrary $x\in \Gr_0$, it suffices to consider a fixed $x\in \cP_m$.
Suppose $y\in \cP_n$, and we write $\widehat{y}$ for it image in $L^2(\Gr_0, \tr_0)$. 
We need only treat the case $m<n$, since we may ignore the behavior of $[N, x]$ on a finite dimensional subspace.
We have
\begin{align*}
[N, x]\widehat{y} 
&= 
N(x\star y)^{\widehat{\,\,}} -x\star(N y)^{\widehat{\,\,}}
\\&= 
\sum_{j=0}^{m}
(m+n - 2j)
\begin{tikzpicture}[baseline=.2cm]
    \draw (-.15,0) -- (-.15,.7);
    \draw (.95,0) -- (.95,.7);
    \draw (.15,.3) arc (180:0:.25cm);
    \roundNbox{unshaded}{(0,0)}{.3}{0}{0}{$x$}
    \roundNbox{unshaded}{(.8,0)}{.3}{0}{0}{$y$}
    \node at (-.6,.5) {\scriptsize{$m-j$}};
    \node at (1.35,.5) {\scriptsize{$n-j$}};
    \node at (.4,.75) {\scriptsize{$j$}};
\end{tikzpicture}
-
\sum_{j=0}^{m}
n
\begin{tikzpicture}[baseline=.2cm]
    \draw (-.15,0) -- (-.15,.7);
    \draw (.95,0) -- (.95,.7);
    \draw (.15,.3) arc (180:0:.25cm);
    \roundNbox{unshaded}{(0,0)}{.3}{0}{0}{$x$}
    \roundNbox{unshaded}{(.8,0)}{.3}{0}{0}{$y$}
    \node at (-.6,.5) {\scriptsize{$m-j$}};
    \node at (1.35,.5) {\scriptsize{$n-j$}};
    \node at (.4,.75) {\scriptsize{$j$}};
\end{tikzpicture}
\\&=
\sum_{j=0}^{m}
(m- 2j)
\begin{tikzpicture}[baseline=.2cm]
    \draw (-.15,0) -- (-.15,.7);
    \draw (.95,0) -- (.95,.7);
    \draw (.15,.3) arc (180:0:.25cm);
    \roundNbox{unshaded}{(0,0)}{.3}{0}{0}{$x$}
    \roundNbox{unshaded}{(.8,0)}{.3}{0}{0}{$y$}
    \node at (-.6,.5) {\scriptsize{$m-j$}};
    \node at (1.35,.5) {\scriptsize{$n-j$}};
    \node at (.4,.75) {\scriptsize{$j$}};
\end{tikzpicture}
.
\end{align*}
We now see we can write this sum at the end as
$$
\left(
\sum_{j=0}^m 
(m-2j)
\begin{tikzpicture}[baseline = .2cm]
	\draw (-.2,.3)--(-.2,.7);	
	\draw (.2,.3) -- (.2,.7);	
	\roundNbox{unshaded}{(0,0)}{.3}{0}{0}{$x$}
	\node at (-.65,.5) {{\scriptsize{$m-j$}}};
	\node at (.4,.5) {{\scriptsize{$j$}}};
	\filldraw[red] (0,.3) circle (.05cm);
\end{tikzpicture}
\right)
\widehat{y}
$$
where the sum in parentheses is a finite sum of bounded operators in the $\cP_\bullet$-Toeplitz algebra $\cT_0(\cP_\bullet)$ \cite{MR3624399}, which is independent of $y$.
We are finished.
\end{proof}

\begin{prop}
The number operator $N$ has compact resolvent and is $\theta$-summable.
\end{prop}
\begin{proof}
We must show that $e^{-t N^2}$ is trace class for all $t>0$.
Since $\Gamma$ is the principal graph of $\cP_\bullet$, $\dim(\cP_n)$ is the number of loops of length $2n$ on $\Gamma$ based at $*$.
Letting $A_\Gamma$ be the adjacency matrix of $\Gamma$, we have that $A_\Gamma$ acts on $\ell^2(V)$, and $\|A_\Gamma\| \leq \delta$ by \cite[\S1.3.5]{MR1278111}.
Define $e_*\in \ell^2(V)$ by $e_*(v) = \delta_{v=*}$, and note that the number of loops based at $*$ of length $2n$ is $\langle A_\Gamma^{2n} e_*, e_*\rangle$.
Hence, we see
$$
\dim(\cP_n) 
= 
\langle A_\Gamma^{2n} e_*, e_*\rangle 
= 
|\langle A_\Gamma e_*, e_*\rangle | 
\leq 
\|A_\Gamma\|^{2n} 
\leq 
\delta^{2n}.
$$
Thus, on $\cP_{n}$, $e^{-t N^2}$ has trace bounded above by 
$$
\dim(\cP_n) e^{-t n^2} \leq \delta^{2n} e^{-t n^2}.
$$
It is now clear by the root test that
\begin{align*}
\sum_{n\geq 0} \frac{\delta^{2n}}{e^{t n^2}} &<\infty. 
\qedhere
\end{align*}
\end{proof}

We now use techniques from \cite{MR3118393} to show that the number operator arises as $\partial^*\partial$ where $\partial$ is a derivation fom $\Gr_0$ into a Hilbert space.

\begin{defn}
We define $\cK_0 = \bigoplus_{i,j\geq 0} \cP_{i,j,1}$, where $\cP_{i,j,1}=\cP_{i+j+1}$, but we think of elements as having $i$ strings out the top, $j$ strings out the bottom, and one string to the right:
$$
\cP_{i,j,1}\ni x
\longleftrightarrow
\begin{tikzpicture}[baseline=-.1cm]
    \draw (0,-.7) -- (0,.7);
    \draw (0,0) -- (.7,0);
    \roundNbox{unshaded}{(0,0)}{.3}{0}{0}{$x$}
    \node at (.2,.5) {\scriptsize{$i$}};
    \node at (.2,-.5) {\scriptsize{$j$}};
\end{tikzpicture}
\,.
$$
We define $\cK$ to be the completion of $\cK_0$ using the inner product 
$$
\left\langle 
\begin{tikzpicture}[baseline=-.1cm]
    \draw (0,-.7) -- (0,.7);
    \draw (0,0) -- (.7,0);
    \roundNbox{unshaded}{(0,0)}{.3}{0}{0}{$x$}
    \node at (.2,.5) {\scriptsize{$i$}};
    \node at (.2,-.5) {\scriptsize{$j$}};
\end{tikzpicture}
\,,\,
\begin{tikzpicture}[baseline=-.1cm]
    \draw (0,-.7) -- (0,.7);
    \draw (0,0) -- (.7,0);
    \roundNbox{unshaded}{(0,0)}{.3}{0}{0}{$y$}
    \node at (.2,.5) {\scriptsize{$k$}};
    \node at (.2,-.5) {\scriptsize{$\ell$}};
\end{tikzpicture}
\right\rangle
=
\delta_{i=k}\delta_{j=\ell}\,
\begin{tikzpicture}[baseline=-.1cm]
    \draw (0,.3) arc (180:0:.4cm);
    \draw (0,-.3) arc (-180:0:.4cm);
    \draw (0,0) -- (.8,0);
    \roundNbox{unshaded}{(0,0)}{.3}{0}{0}{$x$}
    \roundNbox{unshaded}{(.8,0)}{.3}{0}{0}{$y^*$}
    \node at (-.2,.5) {\scriptsize{$i$}};
    \node at (-.2,-.5) {\scriptsize{$j$}};
    \node at (-.5,0) {$\star$};
    \node at (1.3,0) {$\star$};
\end{tikzpicture}
\,.
$$
We have an action of $\Gr_0\otimes \Gr_0^{\text{op}}$ on $\cK$ by bounded operators.
Given $x\in \cP_m$ and $y\in \cP_n$, we think of $x\otimes y^{\text{op}}$ as the following diagram:
$$
x\otimes y^{\text{op}}
=
\begin{tikzpicture}[baseline=-.1cm]
    \draw (0,.4) -- (0,1.1);
    \draw (0,-.4) -- (0,-1.1);
    \roundNbox{unshaded}{(0,.4)}{.3}{0}{0}{$x$}
    \roundNbox{unshaded}{(0,-.4)}{.3}{0}{0}{\rotatebox{180}{$y$}}
    \node at (-.2,.9) {\scriptsize{$m$}};
    \node at (-.2,-.9) {\scriptsize{$n$}};
\end{tikzpicture}
$$
which acts on $\cK$ by left multiplication in a variation of the Bacher-Walker product.
If $z\in \cP_{k,\ell,1}$, we have
$$
(x\otimes y^{\text{op}}) \cdot z 
=
\sum_{i=0}^{\min\{m,k\}}
\sum_{j=0}^{\min\{n,\ell\}}
\,
\begin{tikzpicture}[baseline=-.1cm]
    \draw (-.15,.4) -- (-.15,1.1);
    \draw (-.15,-.4) -- (-.15,-1.1);
    \draw (.95,1.1) -- (.95,-1.1);
    \draw (.8,0) -- (1.5,0);
    \draw (.15,.7) arc (180:0:.25cm) -- (.65,0);
    \draw (.15,-.7) arc (-180:0:.25cm) -- (.65,0);
    \roundNbox{unshaded}{(0,.4)}{.3}{0}{0}{$x$}
    \roundNbox{unshaded}{(0,-.4)}{.3}{0}{0}{\rotatebox{180}{$y$}}
    \roundNbox{unshaded}{(.8,0)}{.3}{0}{0}{$z$}
    \node at (-.55,.9) {\scriptsize{$m-i$}};
    \node at (-.55,-.9) {\scriptsize{$n-j$}};
    \node at (1.3,.9) {\scriptsize{$k-i$}};
    \node at (1.3,-.9) {\scriptsize{$\ell-j$}};
    \node at (.4,1.1) {\scriptsize{$i$}};
    \node at (.4,-1.1) {\scriptsize{$j$}};
\end{tikzpicture}
\,.
$$
It is easy to see that this action is bounded using the Fock space argument of \cite{MR3624399}.
This means that $\cK$ has the natural structure of a $\Gr_0-\Gr_0$ bimodule.
We use the notation $x\star \xi \star y = (x\otimes y^{\text{op}}) \cdot \xi$.
\end{defn}

\begin{defn}
We define a map $\partial : \Gr_0 \to \cK$ by the linear extension of
$$
\cP_m
\ni
\begin{tikzpicture}[baseline=-.1cm]
    \draw (0,0) -- (0,.7);
    \roundNbox{unshaded}{(0,0)}{.3}{0}{0}{$x$}
    \node at (.2,.5) {\scriptsize{$m$}};
\end{tikzpicture}
\longmapsto
\sum_{j=0}^{m-1}
\begin{tikzpicture}[baseline=-.1cm]
    \draw (0,-.7) -- (0,.7);
    \draw (0,0) -- (.7,0);
    \roundNbox{unshaded}{(0,0)}{.3}{0}{0}{$x$}
    \node at (.2,.5) {\scriptsize{$j$}};
    \node at (.7,-.5) {\scriptsize{$m-j-1$}};
\end{tikzpicture}
=
\sum_{j=0}^{m-1}
\begin{tikzpicture}[baseline=-.1cm]
    \draw (0,-.7) -- (0,.7);
    \draw (0,0) -- (.7,0);
    \roundNbox{unshaded}{(0,0)}{.3}{0}{0}{$x$}
    \node at (.2,-.5) {\scriptsize{$j$}};
    \node at (.7,.5) {\scriptsize{$m-j-1$}};
\end{tikzpicture}
\,.
$$
\end{defn}

\begin{lem}
The map $\partial$ is a closable derivation in the Bacher-Walker product.
\end{lem}
\begin{proof}
First, we show that $\partial$ is a derivation.
We need to show that $\partial (x\star y)= x\star \partial (y) + \partial(x) \star y$.
It is straightforward to compute
\begin{align*}
\partial\left(
\sum_{i=0}^{\min\{m,n\}}
\,
\begin{tikzpicture}[baseline=.2cm]
    \draw (-.15,0) -- (-.15,.7);
    \draw (.95,0) -- (.95,.7);
    \draw (.15,.3) arc (180:0:.25cm);
    \roundNbox{unshaded}{(0,0)}{.3}{0}{0}{$x$}
    \roundNbox{unshaded}{(.8,0)}{.3}{0}{0}{$y$}
    \node at (-.55,.5) {\scriptsize{$m-i$}};
    \node at (1.3,.5) {\scriptsize{$n-i$}};
    \node at (.4,.7) {\scriptsize{$i$}};
\end{tikzpicture}
\,\right)
&=
\sum_{i=0}^{\min\{m,n\}}
\sum_{j=0}^{n-i}
\begin{tikzpicture}[baseline=-.1cm]
    \draw (-.15,.4) -- (-.15,1.1);
%    \draw (-.15,-.4) -- (-.15,-1.1);
    \draw (.95,1.1) -- (.95,-1.1);
    \draw (.8,0) -- (1.5,0);
    \draw (.15,.7) arc (180:0:.25cm) -- (.65,0);
%    \draw (.15,-.7) arc (-180:0:.25cm) -- (.65,0);
    \roundNbox{unshaded}{(0,.4)}{.3}{0}{0}{$x$}
%    \roundNbox{unshaded}{(0,-.4)}{.3}{0}{0}{\rotatebox{180}{$y$}}
    \roundNbox{unshaded}{(.8,0)}{.3}{0}{0}{$y$}
    \node at (-.55,.9) {\scriptsize{$m-i$}};
%   \node at (-.55,-.9) {\scriptsize{$n-j$}};
    \node at (1.1,.9) {\scriptsize{$j$}};
    \node at (0,-.9) {\scriptsize{$n-j-i-1$}};
    \node at (.4,1.1) {\scriptsize{$i$}};
%    \node at (.4,-1.1) {\scriptsize{$j$}};
\end{tikzpicture}
+
\sum_{i=0}^{\min\{m,n\}}
\sum_{j=0}^{m-j}
\begin{tikzpicture}[baseline=-.1cm, yscale=-1]
    \draw (-.15,.4) -- (-.15,1.1);
%    \draw (-.15,-.4) -- (-.15,-1.1);
    \draw (.95,1.1) -- (.95,-1.1);
    \draw (.8,0) -- (1.5,0);
    \draw (.15,.7) arc (180:0:.25cm) -- (.65,0);
%    \draw (.15,-.7) arc (-180:0:.25cm) -- (.65,0);
    \roundNbox{unshaded}{(0,.4)}{.3}{0}{0}{\rotatebox{180}{$y$}}
%    \roundNbox{unshaded}{(0,-.4)}{.3}{0}{0}{\rotatebox{180}{$y$}}
    \roundNbox{unshaded}{(.8,0)}{.3}{0}{0}{$x$}
    \node at (-.55,.9) {\scriptsize{$n-i$}};
%   \node at (-.55,-.9) {\scriptsize{$n-j$}};
    \node at (1.1,.9) {\scriptsize{$j$}};
    \node at (0,-.9) {\scriptsize{$m-j-i-1$}};
    \node at (.4,1.1) {\scriptsize{$i$}};
%    \node at (.4,-1.1) {\scriptsize{$j$}};
\end{tikzpicture}
\end{align*}
The right hand side is easily seen to be equal to 
$x\star \partial(y) + \partial(x)\star y$ 
after switching the order of summation.
To show $\partial$ is closable, it is easy to calculate that in the Bacher-Walker product,
$$
\partial^*\left(
\begin{tikzpicture}[baseline=-.1cm]
    \draw (0,-.7) -- (0,.7);
    \draw (0,0) -- (.7,0);
    \roundNbox{unshaded}{(0,0)}{.3}{0}{0}{$x$}
    \node at (.2,.5) {\scriptsize{$j$}};
    \node at (.7,-.5) {\scriptsize{$m-j-1$}};
\end{tikzpicture}
\right)
=
\begin{tikzpicture}[baseline=-.1cm]
    \draw (0,0) -- (0,.7);
    \roundNbox{unshaded}{(0,0)}{.3}{0}{0}{$x$}
    \node at (.2,.5) {\scriptsize{$m$}};
\end{tikzpicture}\,.
$$
Hence $\partial^*$ defined, and we are finished.
\end{proof}

The following corollary is now immediate.

\begin{cor}
The number operator $N=\partial^*\partial$.
\end{cor}

%%%%%%%%%%%%%%%%%%%%%%%%%%%%%%%%%%%%%%%%%%%%%%%%%%%%%%%%%%%
%%%%%%%%%%%%%%%%%%%%%%%%%%%%%%%%%%%%%%%%%%%%%%%%%%%%%%%%%%%
%%%%%%%%%%%%%%%%%%%%%%%%%%%%%%%%%%%%%%%%%%%%%%%%%%%%%%%%%%%
\bibliographystyle{alpha}
{\footnotesize{
\bibliography{bibliography}
%\bibliography{bibliography}
}}
\end{document}